\documentclass{article}
\usepackage[utf8]{inputenc}
\usepackage{amsmath, amssymb, amsthm,mathrsfs}
\usepackage[all]{xy}    
\usepackage[dvipdfmx]{graphicx}   

\newcommand{\ZZ}{\mathbb Z}
\newcommand{\RR}{\mathbb R}
\newcommand{\CC}{\mathbb C}
\renewcommand{\AA}{\mathbb A}
\newcommand{\PP}{\mathbb P}

\newcommand{\bsym}{\boldsymbol}

\newcommand{\GL}{\mathrm{GL}}
\newcommand{\SL}{\mathrm{SL}}
\newcommand{\Spec}{\mathrm{Spec}\,}
\newcommand{\Sing}{\mathrm{Sing}}
\newcommand{\symm}{\mathfrak S}
\newcommand{\diag}{\mathrm{diag}}
\newcommand{\dprime}{{\prime\prime}}

\newcommand*{\angles}[1]{\left\langle #1 \right\rangle}
\newcommand*{\floor}[1]{\left\lfloor #1\right\rfloor}

\newtheorem{lemm}{Lemma}[section]
\newtheorem{thm}[lemm]{Theorem}
\newtheorem{prop}[lemm]{Proposition}
\newtheorem{coro}[lemm]{Corollary}

\title{Crepant resolutions of quotient varieties in positive characteristics and their Euler characteristics}
\author{Takahiro Yamamoto\thanks{Affiliation: Osaka City University Advanced Mathematical Institute (3-3-138, Sugimoto, Sumiyoshi-Ku, Osaka city, Osaka-fu),
email:mass.11235813@gmail.com}}
\date{}

\begin{document}

\maketitle

\begin{abstract}
In characteristic zero, if a quotient variety has a crepant resolution, the Euler characteristic of the crepant resolution is equal to the number of conjugacy classes of the acting group, by Batyrev's theorem. This is one of the McKay correspondence. It is natural to consider the analogue statement in the positive characteristic. In this paper, we present sequences of crepant resolutions of quotient varieties in the positive characteristic and show that one of the sequences gives a counterexample to the analogue statement of Batyrev's theorem.
\end{abstract}
\tableofcontents

\section{Introduction}
Our study begins with the following theorem proved by Batyrev:
\begin{thm}[\cite{Bat}]\label{Bat}
Let \(G\subset\SL_d(\CC)\) be a finite subgroup. Suppose there exists a crepant resolution \(\widetilde X\to X\) of the quotient variety \(X=\AA^3_{\CC}/G\). Then the topological Euler characteristic \(e(\widetilde X)\) of \(\widetilde X\) is equal to the number of conjugacy classes of \(G\):
\[
e(\widetilde X)=\sharp\mathrm{Conj}(G)
\]
\end{thm}
where a crepant resolution is one whose relative canonical divisor is zero.

Thus, it is expected to know if the same statement holds in a positive characteristic when replacing the topological Euler characteristic \(e\) by the \(l\)-adic one \(\chi\), which is defined by the alternating sum of the dimensions of the \(l\)-adic \'etale cohomology with compact support. Note that in characteristic zero the \(l\)-adic Euler characteristic \(\chi\) is equal to the topological Euler characteristic. 

In the positive characteristic, the following results are known.
\begin{itemize}
    \item When we consider the \(E_8^2\) singularity in characteristic two in Artin \cite{Art}, it has double cover by \(A_0\) singularity, and the inverse image of its origin has the Euler characteristic eight. This is a counterexample of Batyrev's theorem for the inverse image of the origin, which claims that the Euler characteristic of the inverse image of the origin along a crepant resolution is equal to the number of conjugacy classes in the case wherein a group action is non-linear and the characteristic is positive.
    
    \item When the order of an acting group is not divided by the characteristic of a base field, the analogue statement of the Theorem \ref{Bat} can be obtained. 
    
    \item If the symmetric group \(\symm_n\) of degree \(n\) acts on \(\AA^{2n}=(\AA^2)^n\) canonically, then the analogue statement of Theorem \ref{Bat} is holds.
    
    \item We consider the subgroup of all signed permutation matrices of \(\SL_n(k)\) in characteristic \(p\geq 3\), where a signed permutation matrix is the one that has only one non-zero entry in each line or column, and the nonzero element is \(\pm 1\). This subgroup acts on \(\AA^{2n}\) canonically. This case and the above one is followed from Kadlaya's formula in Wood and Yasuda \cite{W-Y}.
    
    \item In characteristic \(p\geq 3\), then the group \(\ZZ/p\ZZ\) has a linear action on \(\AA^3\) without pseudo-reflection. Yasuda \cite{Yas} shows that, if \(p=3\), then the quotient variety with respect to this action has a crepant resolution, and the equation in Batyrev's theorem holds. Otherwise, crepant resolutions of the quotient varieties do not exist. 
\end{itemize}

We consider canonical actions of a \textit{wild} and \textit{small} subgroups of \(\SL_n(k)\), that is a subgroup of \(\SL_n(k)\), whose order is divided by characteristic without pseudo-reflections. If \(n=2\), any wild subgroup of \(\SL_n(k)\) has a pseudo-reflection. Hence the smallest dimension we interested in is three. From \cite{Yas}, if \(p\geq 5\), then the quotient variety does not have a crepant resolution in dimension three. If \(p=2\), the quotient variety in dimension three is nonsingular. Hence, we consider the case that the characteristic \(p=3\) and dimension three.

In this case, we obtain sequences of quotient varieties with crepant resolutions. Furthermore, we show that a quotient variety in one of these sequences holds an analogue statement of Batyrev's theorem, but a quotient variety in the other does not. \begin{thm}[Main Theorem \ref{main}]
Let \(k\) be an algebraically closed field of characteristic three. Let \(G\) be a small finite subgroup of \(\SL_3(k)\) which is a semidirect product \(H\rtimes G^\prime\) of a tame Abelian group \(H\) and a cyclic group \(G^\prime\cong\ZZ/3\ZZ\) or the symmetric group \(G^\prime\cong\symm_3\). If \(G\) acts on the affine space \(\AA^3_k\) canonically, then the quotient variety \(X=\AA^3_k/G\) has a crepant resolution \(\widetilde X\to X\). Moreover, we have the following equalities:
\[
\chi(\widetilde X)=\begin{cases}
\sharp\mathrm{Conj}(G)&(G^\prime\cong\ZZ/3\ZZ)\\
\sharp\mathrm{Conj}(G)+3&(G^\prime\cong\symm_3)
\end{cases}
\]
\end{thm}
Chen-Du-Gao present an example of a quotient variety, which does not hold the analogue statement of Batyrev's theorem in \cite[Example in page 11]{CDG}. In the example, the acting group has a pseudo-reflection. A part of the main theorem, including a counterexample of Batyrev's theorem, is shown in my master's thesis \cite[Theorem 1.2]{Yama}, and the main theorem is presented in my doctoral thesis \cite[Theorem 1.3]{Yama2}. In the doctoral thesis \cite[Corollary 1.7]{Yama2}, we show the same assertion as one of the main theorems for the case wherein the maximal order of elements in \(H\) is prime by stringy-point count. 

We show the main theorem by direct computation. First, we show that the given group of \(\SL_3(k)\) can be transformed to a subgroup of some simple form by conjugation in \(\GL_3(k)\). Next, we show the main theorem in the case that \(H\) is trivial by direct computation. In the general case, we compute the crepant resolution by combining the result of the previous case and the theory of toric varieties. 

The outline of the paper is as follows. In section three, we determine the generator of the acting group \(G\) for each case up to conjugate. In section four, we present the main theorem for the fundamental case in which \(H\) is trivial. In section five, we present the main theorem. 

\subsection*{Acknowledgements}
I am grateful to Takehiko Yasuda for teaching me and for his valuable comments. Without his help, this paper would not have been possible. This work was partially supported by JSPS KAKENHI Grant Number JP18H01112.

\section{Notation and Convention}
In this paper, we denote a base field by \(k\) and suppose that \(k\) is an algebraically closed field of characteristic three. A variety means a separated scheme of finite type over \(k\). 

For a finite subgroup \(G\) of \(\GL_d(k)\), we refer to \(G\) as \textit{wild} if the order of \(G\) is divisible by the characteristic of \(k\). Otherwise, we say \(G\) is \textit{tame}. A \textit{small} subgroup of \(\GL_d(k)\) is a subgroup without pseudo-reflections, which are matrices whose stable subspace has dimension \(d-1\).
\section{Group action}
In this section, we determine a small subgroup \(G\) of \(\SL_3(k)\), such that \(G\) have a tame Abelian normal subgroup \(H\) and \(G/H\) is isomorphic to \(\ZZ/3\ZZ\) or \(\symm_3\). Note that \(G\) and \(\bsym XG\bsym X^{-1}\) for \(\bsym X\in\GL_3(k)\) are subgroups of \(\SL_3(k)\), and these define isomorphic representations. Hence, we determine such group up to conjugation with elements of \(\GL_3(k)\).

As \(H\) is a tame Abelian group, we can diagonalize all the elements of \(H\) by the same matrix. Hence, we suppose that all elements of \(H\) are diagonal matrices. As \(G\) is wild, there exists an element \(\bsym S\in G\) of order three. The matrix \(\bsym S\) is not pseudo-reflection. Hence, \(\bsym S\) conjugate to
\[\begin{bmatrix}
1&1&0\\
0&1&1\\
0&0&1
\end{bmatrix}\]
and 
\[\begin{bmatrix}
0&1&0\\
0&0&1\\
1&0&0
\end{bmatrix}.\]
\begin{lemm}\label{lem:gr1}
There exists a matrix \(\bsym M\in\GL_3(k)\) such that \(\bsym MH\bsym M^{-1}\) consists of diagonal matrices and
\[\begin{bmatrix}
0&1&0\\
0&0&1\\
1&0&0
\end{bmatrix}\in\bsym MG\bsym M^{-1}\]
\end{lemm}
\begin{proof}
If \(H\) is trivial, existence of the matrix \(\bsym M\) is obvious. We assume that \(H\) is not trivial.

Take \(I\ne\bsym X\in H\). We write \(\bsym X=\diag(x_1,x_2,x_3)\). As \(H\) is a normal subgroup of \(G\), we can write \(\bsym{SXS}^{-1}=\diag(y_1,y_2,y_3)\). Set \(\bsym S=[s_{ij}]\). Then, we have
\[[s_{ij}x_j]=[y_is_{ij}].\]
If \(x_j\not\in \{y_1,y_2,y_3\}\), \(s_{ij}=0\) for \(i=1,2,3\), which contradicts that \(\bsym S\) is invertible. Hence \(\{x_1,x_2,x_3\}\subset\{y_1,y_2,y_3\}\). We obtain the inverse inclusion similar argument. Therefore \(\{x_1,x_2,x_3\}=\{y_1,y_2,y_3\}\). If \(x_1=x_2=x_3\), we have
\[\det\bsym X=x_1x_2x_3=x_1^3=1\]
because \(\bsym X\in\SL_3(k)\). This implies \(x_1=1\) and hence \(\bsym X=I\), which contradicts the assumption \(\bsym X\ne I\). Hence, by permuting coordinates if necessary, we may assume that \(x_1\ne x_2\) and \(x_1\ne x_3\).

If \(x_1=y_1\), then \(s_{12}=s_{13}=s_{21}=s_{31}=0\). Hence
\[\bsym S=\begin{bmatrix}
1&0\\
0&\bsym S^\prime
\end{bmatrix}\]
where \(\bsym S^\prime\) is a matrix in \(\SL_2(k)\) of order three. This implies that \(\bsym S\) is a pseudo-reflection, which is impossible as \(G\) is small. Therefore \(x_1\ne y_1\), and hence \(x_1=y_2\) or \(x_1=y_3\). We may assume \(x_1=y_2\). Then
\[
\bsym S=\begin{bmatrix}
0&s_{12}&s_{13}\\
1&0&0\\
0&s_{32}&s_{33}
\end{bmatrix}.
\]
We define \(\bsym Z=\bsym S\bsym Y\bsym S^{-1}\) and we write \(\bsym Z=[z_{ij}]\). We have \(\{y_1,y_2,y_3\}=\{z_1,z_2,z_3\}\). Comparing the (2, 1) elements of
\[
\bsym S\bsym Y=\begin{bmatrix}
0&s_{12}y_2&s_{13}y_3\\
y_1s_{21}&0&0\\
0&s_{32}y_2&s_{33}y_3
\end{bmatrix}
\]
and
\[\bsym Z\bsym S=\begin{bmatrix}
0&z_1s_{12}&z_1s_{13}\\
z_2s_{21}&0&0\\
0&z_3s_{32}&z_3s_{33}
\end{bmatrix}\]
we have \(y_1=z_2\). As \(x_1\ne x_2,x_3\), we have \(x_1=y_2\ne y_1=z_2\). As \(\bsym S^3=I\), \(\bsym S\bsym Z\bsym S^{-1}=\bsym X\). Comparing the (2,1) elements of \(\bsym {SZ}\) and \(\bsym {XS}\), we obtain \(z_1=x_2\ne x_1=y_2\). Therefore, we obtain
\[\bsym S=\begin{bmatrix}
0&0&s_{13}\\
s_{21}&0&0\\
0&s_{32}&0
\end{bmatrix}.\]
Here, \(\det\bsym S=s_{13}s_{21}s_{32}=1\). Hence
\[\diag(s_{21},s_{21}s_{32},1)\bsym S\diag(s_{21},s_{21}s_{32},1)^{-1}=\begin{bmatrix}
0&0&1\\
1&0&0\\
0&1&0
\end{bmatrix}.\]
Therefore, \(\bsym M=\diag(s_{21},s_{21}s_{32},1)\) satisfies the desired property.
\end{proof}

From Lemma \ref{lem:gr1}, we may assume that the normal subgroup \(H\) of \(G\) is diagonal and
\[
\bsym S=\begin{bmatrix}
0&0&1\\
1&0&0\\
0&1&0
\end{bmatrix}\in G.
\]

\begin{lemm}\label{lem:gr2}
If \(G\cong H\rtimes\symm_3\), then there exists a matrix \(\bsym M\in\GL_3(k)\) satisfying the properties in Lemma \ref{lem:gr1} and
\[
\begin{bmatrix}
0&0&-1\\
0&-1&0\\
-1&0&0
\end{bmatrix}\in\bsym M G\bsym M^{-1}
\]
\end{lemm}
\begin{proof}
If \(H\) is trivial, we consider \(\bsym T\in G\) such that
\[
\bsym T^2=I,\ \bsym T\bsym S=\bsym S^2\bsym T.
\]
We substitute \(\bsym T=[\bsym t_j]\) where \(\bsym t_j\) are vertical vectors. Then
\[
[\bsym t_2,\bsym t_3,\bsym t_1]=\bsym T\bsym S=\bsym S^2\bsym T=[\bsym S^2\bsym t_1,\bsym S^2\bsym t_2,\bsym S^2\bsym t_3].
\]
Hence
\[
\bsym t_2=\bsym S^2\bsym t_1,\ \bsym t_3=\bsym S\bsym t_1.
\]
Substituting \(\bsym t_1=[a,b,c]^t\),
\[
\bsym T=\begin{bmatrix}
a&b&c\\
b&c&a\\
c&a&b
\end{bmatrix}.
\]

We consider the case \(H\) is trivial.
As
\[
\det\bsym T=-(a^3+b^3+c^3)=-1,
\]
we obtain
\[a+b+c=-1.\]
We have
\[
\bsym T^2=\begin{bmatrix}
a^2+b^2+c^2&ab+bc+ca&ab+bc+ca\\
ab+bc+ca&a^2+b^2+c^2&ab+bc+ca\\
ab+bc+ca&ab+bc+ca&a^2+b^2+c^2
\end{bmatrix}.
\]
As \(\bsym T^2=I\), we obtain \(a^2+b^2+c^2=1\). Combining these equations, we obtain
\begin{align*}
a^2+b^2+(-1-a-b)^2&=1\\
\Leftrightarrow\ a^2+ab+b^2+a+b&=0.
\end{align*}
This is a quadratic equation for \(b\). Solving this equation, we obtain
\[
b=(a+1)+\alpha
\]
where \(\alpha\) is an element of \(k\) that satisfies \(\alpha^2=a+1\). Then we can write
\[
a=\alpha^2-1,\ b=\alpha^2+\alpha,\ c=\alpha^2-\alpha.
\]
Let
\[\bsym P=(\alpha-1)I-(\alpha+1)\bsym S.\]
Then
\begin{align*}
\bsym P^3&=(\alpha-1)^3I-(\alpha+1)^3\bsym S^3\\
&=(\alpha^3-1)I-(\alpha^3+1)I\\
&=I.
\end{align*}
We have
\begin{align*}
\bsym P\bsym T\bsym P^{-1}&=\bsym P\bsym T\bsym P^2\\
&=((\alpha-1)I-(\alpha+1)\bsym S)\bsym T((\alpha-1)^2I+(\alpha^2-1)\bsym S+(\alpha+1)^2\bsym S^2)\\
&=\bsym T((\alpha-1)I-(\alpha+1)\bsym S^2)((\alpha-1)^2I+(\alpha^2-1)\bsym S+(\alpha+1)^2\bsym S^2)\\
&=\bsym T(-\alpha(\alpha-1)I-\alpha(\alpha+1)\bsym S-(\alpha^2-1)\bsym S^2)\\
&=-\bsym T(cI+b\bsym S+a\bsym S^2).
\end{align*}
Here 
\[
cI+b\bsym S+a\bsym S^2=\begin{bmatrix}
c&b&a\\
a&c&b\\
b&a&c
\end{bmatrix}=\bsym T\begin{bmatrix}
0&0&1\\
0&1&0\\
1&0&0
\end{bmatrix}.
\]
Hence,
\[
\bsym P\bsym T\bsym P^{-1}=-\bsym T^2\begin{bmatrix}
0&0&1\\
0&1&0\\
1&0&0
\end{bmatrix}=\begin{bmatrix}
0&0&-1\\
0&-1&0\\
-1&0&0
\end{bmatrix}.
\]
Therefore, \(\bsym M=\bsym P\) satisfies the desired properties.

Next, we consider the case where \(H\) is not trivial. Consider \(I\ne X=\diag(x_1,x_2,x_3)\in H\). We may assume that \(x_1\ne x_2\) and \(x_1\ne x_3\). We define \(\bsym Y=\bsym T\bsym X\bsym T\). We write \(\bsym Y=\diag(y_1,y_2,y_3)\). Then \(\{x_1,x_2,x_3\}=\{y_1,y_2,y_3\}\) since \(\bsym T\) is invertible. Hence \(x_1=y_1,\ y_2,\) or \(y_3\). This implies that
\[
T=\begin{bmatrix}
a&0&0\\
0&0&a\\
0&a&0
\end{bmatrix},\ \begin{bmatrix}
0&a&0\\
a&0&0\\
0&0&a
\end{bmatrix},\text{ or }\begin{bmatrix}
0&0&a\\
0&a&0\\
a&0&0
\end{bmatrix},
\]
respectively. For each case, we have \(\det\bsym T=-a^3=1\). Hence, \(a=-1\). As
\[
\bsym S^{-1}\begin{bmatrix}
-1&0&0\\
0&0&-1\\
0&-1&0
\end{bmatrix}\bsym S=\bsym S\begin{bmatrix}
0&-1&0\\
-1&0&0\\
0&0&-1
\end{bmatrix}\bsym S^{-1}=\begin{bmatrix}
0&0&-1\\
0&-1&0\\
-1&0&0
\end{bmatrix},
\]
we obtain
\[
\begin{bmatrix}
0&0&-1\\
0&-1&0\\
-1&0&0
\end{bmatrix}\in G.
\]
Hence, \(\bsym M=\bsym I\) satisfies the desired properties.
\end{proof}

From Lemma \ref{lem:gr2}, when \(G\cong H\rtimes\symm_3\) we assume that \(H\) is diagonal and \(G\) contains
\[
\bsym S=\begin{bmatrix}
0&0&1\\
1&0&0\\
0&1&0
\end{bmatrix},\ 
\bsym T=\begin{bmatrix}
0&0&-1\\
0&-1&0\\
-1&0&0
\end{bmatrix}.
\]

Finally, we study the number of conjugacy classes. For subgroup \(F\subset G\), we denote the orbit of \(\bsym X\in G\) about conjugate action by \(O_F(\bsym X)\). We denote the largest order of elements in \(H\) by \(r\). Note that \(r\) is not divisible by three as \(r\) divides \(\sharp H\), and \(H\) is a tame group.
\begin{lemm}\label{lem:conjnum1}
When \(G\cong H\rtimes \ZZ/3\ZZ\), the number of conjugacy classes of \(G\) is \((\sharp H-1)/3+3\).
\end{lemm}
\begin{proof}
As \(G\) is generated by \(H\) and \(\bsym S\), any element \(\bsym X\) is written as \(\bsym X=\bsym Y\bsym S^i\) by some \(\bsym Y\in H\) and \(i=0,1,2\). Hence, for all \(\bsym X\in G\), we have
\[O_G(\bsym X)=\{\bsym S^i\bsym X^\prime\bsym S^{-i}\mid\bsym X^\prime\in O_H(\bsym X),\ i=0,1,2\}.\]

As \(H\) is Abelian, \(O_H(\bsym X)=\{\bsym X\}\) for \(\bsym X\in H\). Hence,
\[O_G(\bsym X)=\{\bsym X,\bsym S\bsym X\bsym S^{-1},\bsym S^{-1}\bsym X\bsym S\}.\]
If \(\bsym X=\diag(x_1,x_2,x_3)\), then 
\[
\bsym S\bsym X\bsym S^{-1}=\bsym X\Leftrightarrow \bsym{SX}=\bsym{XS}\Leftrightarrow\diag(x_3,x_1,x_2)=\diag(x_2,x_3,x_1).
\]
Hence, \(O_G(\bsym X)=\{\bsym X\}\) if and only if \(\bsym X\) is a scalar matrix. As \(H\) is tame, any scalar matrix in \(H\) is the identity matrix. Therefore,
\[
O_G(\bsym X)=\begin{cases}
\{\bsym X,\bsym S\bsym X\bsym S^{-1},\bsym S^{-1}\bsym X\bsym S\}&(\bsym X\ne I)\\
\{\bsym X\}&(\bsym X=I)
\end{cases},
\]
and hence, \((\sharp H-1)/3+1\) conjugacy classes are contained in \(H\).

We show that \(O_G(\bsym S)=H\bsym S\). For \(\bsym X\in H\),
\begin{align*}
\bsym X\bsym S\bsym X^{-1}=\bsym X\bsym S\bsym X^{-1}\bsym S^{-1}\bsym S.
\end{align*}
As \(H\) is a normal subgroup, \(\bsym S\bsym X^{-1}\bsym S^{-1}\in H\). Hence, 
\(\bsym X\bsym S\bsym X^{-1}\bsym S^{-1}\bsym S\in H\bsym S\). Therefore, we obtain \(O_H(X)\subset H\bsym S\). As \(\bsym S^{-1}(H\bsym S)\bsym S= H\bsym S\), we obtain \(O_G(X)\subset H\bsym S\). 

As \(r\) is not divisible by three, we define a positive integer \(d\) such that \(3d\equiv 1\mod r\). For \(\bsym X\in H\), we set \(\bsym Y=(\bsym X(\bsym S^{-1}\bsym X\bsym S)^{-1})^d\in H\). We write \(\bsym X=\diag(x_1,x_2,x_3)\). Note that \(x_1^r=x_2^r=x_3^r=x_1x_2x_3=1\) as \(\bsym X^r=I\) and \(\det X=1\). Then,
\begin{align*}
\bsym Y&=(\diag(x_1,x_2,x_3)\diag(x_2^{-1},x_3^{-1},x_1^{-1}))^d\\
&=\diag\left(\frac{x_1^d}{x_2^d},\frac{x_2^d}{x_3^d},\frac{x_3^d}{x_1^d}\right).
\end{align*}
We obtain
\begin{align*}
\bsym Y\bsym S\bsym Y^{-1}&=\bsym Y(\bsym S\bsym Y^{-1}\bsym S^{-1})\bsym S\\
&=\diag\left(\frac{x_1^d}{x_2^d},\frac{x_2^d}{x_3^d},\frac{x_3^d}{x_1^d}\right)\diag\left(\frac{x_3^d}{x_1^d},\frac{x_1^d}{x_2^d},\frac{x_2^d}{x_3^d}\right)^{-1}\bsym S\\
&=\diag\left(\frac{x_1^{2d}}{x_2^dx_3^d},\frac{x_2^{2d}}{x_3^dx_1^d},\frac{x_3^{2d}}{x_1^dx_2^d}\right)\bsym S
\end{align*}
As \(x_1x_2x_3=1\), we obtain
\[
\bsym Y\bsym S\bsym Y^{-1}=\diag(x_1^{3d},x_2^{3d},x_3^{3d})\bsym S.
\]
Since \(x_1^r=x_2^r=x_3^r=1\) and \(3d\equiv1\mod r\), we get
\[
\bsym Y\bsym S\bsym Y^{-1}=\diag(x_1,x_2,x_3)\bsym S=\bsym X\bsym S.
\]
Therefore, \(H\bsym S\subset O_H(\bsym S)\subset O_G(\bsym S)\), and hence, \(O_G(\bsym S)=H\bsym S\).

Similarly, we obtain \(O_G(\bsym S^2)=H\bsym S^2\). Therefore,
\[
\sharp\mathrm{Conj}(G)=\frac{\sharp H-1}3+3.
\]
\end{proof}

Lastly, we consider the number of conjugacy classes in \(G\cong H\rtimes\symm_3\). Fix a primitive \(r\)-th root \(\zeta_r\in k\). As any \(\bsym X\in H\) satisfies \(\bsym X^r=I\), we can write
\[
\bsym X=\diag(\zeta_r^{x_1},\zeta_r^{x_2},\zeta_r^{x_3})
\] 
by integers \(0\leq x_1,x_2,x_3<r\). Note that \(x_1+x_2+x_3=r\text{ or }2r\) since \[\zeta_r^{x_1+x_2+x_3}=\det\bsym X=1.\] Hence \(x_3\) determines from \(x_1,x_2\). Then we have injective homomorphism of groups
\[
\log:H\to (\ZZ/r\ZZ)^2
\]
defined by
\[
\log(\diag(\zeta_r^{x_1},\zeta_r^{x_2},\zeta_r^{x_3}))=(x_1+r\ZZ,x_2+r\ZZ).
\]
\begin{lemm}\label{lem:aboutH}
When \(G\rtimes\symm_3\), the homomorphism \(\log:H\to(\ZZ/r\ZZ)^2\) is an isomorphism. In particular, 
\[H=\{\diag(\zeta_r^{e_1},\zeta_r^{e_2},\zeta_r^{e_3})\mid 0\leq e_i<r\ (i=1,2,3)\}.\]
\end{lemm}
\begin{proof} 
We have only to show that \(\log\) is surjective. Let \(\bsym X=\diag(x_1,x_2,x_3)\) be an element of \(H\) of order \(r\). We may assume that \(x_2^r=1\) and \(x_2^n\ne 1\) for any \(1\leq n<r\). We put
\[
\bsym X^\prime=\bsym {XTXT}=\diag(x_3x_1,x_2^2,x_1x_3).
\]
Since \(x_1x_2x_3=\det\bsym X=1\), we have
\[
\bsym X^\prime=\diag(x_2^{-1},x_2^2,x_2^{-1}).
\]
As \(x_2^n=1\) for \(1\leq n\leq r\) if and only if \(n=r\), the order of \(\bsym X^\prime\) is \(r\). Replacing \(\bsym X\) by \(\bsym X^\prime\), we may assume that \(x_1=x_3\). Thus, \(x_2=\frac 1{x_1x_3}=x_1^{-2}\). Hence, we can write as follows:
\[
\log\bsym X=(e+r\ZZ,-2e+r\ZZ)
\]
by \(e\in\ZZ\). As the order of \(\bsym X\) is \(r\), \(\gcd(e,r)=\gcd(e,-2e,r)=1\). We have
\[
\log\bsym {SXS}^{-1}=(e+r\ZZ,e+r\ZZ).
\]
SincAse \(r\) is not divisible by three and \(\gcd(e,r)=1\), a determinant
\[
\det\begin{bmatrix}
e&-2e\\
e&e
\end{bmatrix}=3e^2
\]
is coprime to \(r\). Thus the images \(\log\bsym X, \log(\bsym{SXS}^{-1})\) generate \((\ZZ/r\ZZ)^2\).
Therefore, the homomorphism \(\log\) is an isomorphism.
\end{proof}

We define
\begin{gather*}
H_f=\{\diag(x_1,x_2,x_3)\in H\mid x_1\ne x_2\ne x_3\ne x_1\}.
\end{gather*}
Then we have
\[
\sharp(H\backslash H_f)=3(r-1)+1,
\]
and hence,
\[
\sharp H_f=r^2-\sharp(H\backslash H_f)=(r-1)(r-2).
\]
\begin{lemm}\label{lem:conjnum2}
When \(G\cong H\rtimes\symm_3\), the number of conjugacy classes of \(G\) is
\[
\frac{(r-1)(r-2)}6+2r+1.
\]
\end{lemm}
\begin{proof}
The group \(G\) is generated by \(H, \bsym S\), and \(\bsym T\). We denote a subgroup of \(G\) generated by \(H,\bsym S\) by \(H^\prime\). Note that \(H^\prime\) is a normal subgroup of \(G\). As \(G=H^\prime\sqcup H^\prime\bsym T\), we have
\[
O_G(\bsym X)=O_{H^\prime}(\bsym X)\cup \bsym TO_{H^\prime}(\bsym X)\bsym T.
\]

When \(I\ne\bsym X\in H\), from the proof of Lemma \ref{lem:conjnum1},
\[
O_{H^\prime}=\{\bsym x,\bsym{SXS}^{-1},\bsym S^{-1}\bsym{XS}\}.
\]
Then,
\[
O_{H^\prime}(\bsym X)=\bsym TO_{H^\prime}(\bsym X)\bsym T
\]
if and only if \(\bsym X\not\in H_f\). Hence,
\[
\sharp O_G(\bsym X)=\begin{cases}
1&(\bsym X=I)\\
3&(I\ne\bsym X\not\in H_f)\\
6&(\bsym X\in H_f)
\end{cases}
\]
for \(\bsym X\in H\).
Therefore, the number of conjugacy classes in \(H\) is 
\[
\frac{\sharp H_f}6+\frac{\sharp(H\backslash H_f)-1}3+1=\frac{(r-1)(r-2)}6+r.
\]

We have
\[
O_{H^\prime}(\bsym S)=H\bsym S
\]
and
\[
O_{H^\prime}(\bsym S^2)=H\bsym S^2.
\]
As \(\bsym TH\bsym T=H\), we obtain
\[
\bsym TO_{H^\prime}(\bsym S)\bsym T=H\bsym S^2=O_{H^\prime}(\bsym S^2).
\]
Therefore, we obtain
\[
O_G(\bsym S)=H\bsym S\sqcup H\bsym S^2.
\]

Lastly, we compute conjugacy classes in \(H^\prime\bsym T\). Every matrix in \(H^\prime\bsym T\) is written as \(\bsym {XS}^i\bsym T\) by \(\bsym X\in H\) and \(i=0,1,2\). As
\[
\bsym S^{-1}\bsym{TS}=\bsym{ST},
\]
we have
\[
\bsym S^i\bsym{XS}^i\bsym{TS}^{-i}=(\bsym{S}^i\bsym{XS}^{-i})\bsym T\in H\bsym T.
\]
Thus, we can choose a representative of a conjugacy class in \(H^\prime\bsym T\) from \(H\bsym T\). For \(\bsym X, \bsym Y\in H\), we have
\begin{align*}
\bsym Y^{-1}(\bsym{XT})\bsym Y&=(\bsym Y^{-1}\bsym{XTYT})\bsym T.
\end{align*}
If \(\log(\bsym Y)=(y_1+r\ZZ,y_2+r\ZZ), \log(\bsym X)=(x_1+r\ZZ,x_2+r\ZZ)\), then 
\[
\log (\bsym Y^{-1}\bsym X\bsym T\bsym Y\bsym T)=((x_1-y_2)+r\ZZ,x_2+r\ZZ).
\]
Hence, \(\diag(x_1,x_2,x_3)\bsym T\) and \(\diag(x^\prime_1,x_2^\prime,x_3^\prime)\bsym T\) are conjugate if and only if \(x_2=x^\prime_2\). Therefore, \(H^\prime\bsym T\) contains \(r\) conjugacy classes \[O_G(\diag(0,\zeta_r^a,\zeta_r^{-a})\bsym T)\ (a=0,1,\ldots,r-1).\] 

Therefore, \(G\) contains \(\frac{(r-1)(r-2)}6+2r+1\) conjugacy classes.
\end{proof}

\section{Quotient singularity for symmetric group}
We compute the quotient singularity for \(G\cong\symm_3\). We may assume that \(G\) is generated by
\[
\bsym S=\begin{bmatrix}
0&0&1\\
1&0&0\\
0&1&0
\end{bmatrix},\ 
\bsym T=\begin{bmatrix}
0&0&-1\\
0&-1&0\\
-1&0&0
\end{bmatrix}
\]
from Lemma \ref{lem:gr2}.

We consider the action of \(G\) on \(\AA^3_k\) corresponding to the right action of \(G\) on the coordinate ring \(k[x_1,x_2,x_3]\) of \(\AA^3_k\) defined from 
\[
[x_1,x_2,x_3]\cdot \bsym X=[x_1,x_2,x_3]\bsym X.
\]

Our goal in this section is constructing a crepant resolution of \(\AA_k^3/G\) and computing its Euler characteristic. A \textit{crepant resolution} is a resolution \(f:\widetilde X\to X\) whose relative canonical divisor \(K_{\widetilde X/X}\) is zero: \(K_{\widetilde X/X}=K_{\widetilde X}-f^\ast K_X=0\).

\begin{lemm}\label{lem:invRing}
For the above \(G\)-action on \(k[x_1,x_2,x_3]\), the invariant ring \(k[x_1,x_2,x_3]^{\bsym S}\) is generated by 
\begin{gather*}
Y_1=x_1+x_2+x_3,\\
Y_2=x_1x_2+x_2x_3+x_3x_1,\\
Y_3=x_1x_2x_3,\\
Y_4=x_1^2x_2+x_2^2x_3+x_3^2x_1.
\end{gather*}
The invariant ring \(k[x_1,x_2,x_3]^{\bsym S}\) is isomorphic to
\[k[y_1,y_2,y_3,y_4]/(y_4^2+y_2^3+y_1^2y_3-y_1y_2y_4).\]
\end{lemm}
\begin{proof}
From \cite[Theorem 4.10.1]{CW}, we obtain
\[
k[x_1,x_2,x_3]^{\bsym M}=k[x_1,x_2(x_2+x_1)(x_2-x_1),x_3(x_3+x_2)(x_3-x_2+x_1),x_2^2+x_1x_3-x_1x_2]
\]
where
\[
\bsym M=\begin{bmatrix}
1&1&0\\
0&1&1\\
0&0&1
\end{bmatrix}.
\]
We define
\[
\bsym P=\begin{bmatrix}
1&-1&1\\
1&1&0\\
1&0&0
\end{bmatrix}.
\]
Then, \(\bsym P\) satisfies \(\bsym P^{-1}\bsym{MP}=\bsym S\). Hence, an isomorphism
\[\begin{array}{ccc}
k[x_1,x_2,x_3]^{\bsym M}&\to&k[x_1,x_2,x_3]^{\bsym S}  \\
     f&\mapsto& f\cdot P 
\end{array}\]
is well-defined. Therefore, \(k[x_1,x_2,x_3]^{\bsym S}\) is generated by
\begin{gather*}
x_1\cdot P=x_1+x_2+x_3\\
(x_2(x_2+x_1)(x_2-x_1))\cdot P=(x_2-x_1)(x_3-x_2)(x_1-x_3)\\
(x_3(x_3+x_2)(x_3-x_2+x_1))\cdot P=x_1x_2x_3\\
(x_2^2+x_1x_3-x_1x_2)\cdot P=-(x_1x_2+x_2x_3+x_3x_1).
\end{gather*}
Note that
\begin{align*}
(x_2-x_1)(x_3-x_2)(x_1-x_3)=Y_1Y_2+Y_4.
\end{align*}
Hence, \(Y_1, Y_2, Y_3, Y_4\) generate \(k[x_1,x_2,x_3]^{\bsym S}\).

Next, we consider a homomorphism \(\phi: k[y_1,y_2,y_3,y_4]\to k[x_1,x_2,x_3]^{\bsym S}\) defined by \(\phi(y_i)=Y_i\) for \(i=1,2,3,4\) and show that \(\ker\phi=(y_4^2+y_2^3+y_1^2y_3-y_1y_2y_4)\). Defining grading of \(k[y_1,y_2,y_3,y_4]\) by \(\deg y_1=1,\ \deg y_2=2,\ \deg y_3=\deg y_4=3\), the \(\phi\) is a graded ring homomorphism. Then \(\ker\phi\) is the graded ideal. For these graded rings or ideal \(S\), we denote the set of homogeneous elements of degree \(d\) by \(S_d\) and let \(H(S)(\lambda)\) be a formal series \(\sum_{d=0}^\infty\dim_k(S_d)\lambda^d\). It can be observed that
\[
H(k[y_1,y_2,y_3,y_4])(\lambda)=\frac 1{(1-\lambda)(1-\lambda^2)(1-\lambda^3)^2}.
\]
Meanwhile, \(k[x_1,x_2,x_3]^{\bsym S}_d\) is spanned by \(m_{abc}(a\leq b\leq c, a+b+c=d)\), where
\[
m_{abc}=\begin{cases}
x_1^ax_2^ax_3^a&(a=b=c)\\
x_1^ax_2^bx_3^c+x_1^bx_2^cx_3^a+x_1^cx_2^ax_3^b&(otherwise)
\end{cases}
\]
Therefore, 
\[
\dim_kk[x_1,x_2,x_3]^{\bsym S}_d=\begin{cases}
\frac 13\binom{d+2}2&(3\nmid d)\\
\frac 13(\binom{d+2}2-1)+1&(3\mid d)\\
\end{cases}
\]
and we obtain
\begin{align*}
H(k[x_1,x_2,x_3]^{\bsym S})(\lambda)&=\frac 13\left(\frac 1{(1-\lambda)^3}-\frac 1{1-\lambda^3}\right)+\frac 1{1-\lambda^3}\\
&=\frac{1-\lambda^6}{(1-\lambda)(1-\lambda^2)(1-\lambda^3)^2}
\end{align*}
Hence, we obtain
\begin{align*}
H(\ker\phi)(\lambda)&=H(k[y_1,y_2,y_3,y_4])(\lambda)-H(k[x_1,x_2,x_3]^{\bsym S})(\lambda)\\
&=\frac{\lambda^6}{(1-\lambda)(1-\lambda^2)(1-\lambda^3)^2}
\end{align*}
As \(H(\ker\phi)(\lambda)\) coincides with \(H((y_4^2+y_2^3+y_1^3y_3-y_1y_2y_4))(\lambda)\) and \((y_4^2+y_2^3+y_1^3y_3-y_1y_2y_4)\subset\ker\phi\), we obtain \(\ker\phi=((y_4^2+y_2^3+y_1^3y_3-y_1y_2y_4)).\)
\end{proof}
From \ref{lem:invRing}, we obtain a quotient variety
\[
\AA_k^3/\angles{\bsym S}\cong Y:=V(y_4^2+y_2^3+y_1^2y_3-y_1y_2y_4)\subset\AA_k^4.
\]
As
\begin{gather*}
Y_1\cdot\bsym T=-Y_1,\\
Y_2\cdot\bsym T=Y_2,\\
Y_3\cdot\bsym T=-Y_3,\\
Y_4\cdot\bsym T=Y_4-Y_1Y_2,
\end{gather*}
the \(\angles{\bsym T}\)-action on \(Y\) is defined by
\begin{align}
[y_1,y_2,y_3,y_4]\cdot\bsym T=[-y_1,y_2,-y_3,y_4-y_1y_2].\label{eq:actOnY}
\end{align}
\begin{prop}\label{3-prop:resolY}
The quotient \(Y\) has a crepant resolution \(\widetilde Y\to Y\). For further detail, \(\widetilde Y\) is given by gluing three affine varieties \(T_1,T_2,U_2\) defined by
\begin{gather*}
T_1=\AA_k^3,\\
T_2=V(t_{23}^2+t_{21}-t_ut_{21}t_{23})\subset\AA_k^4,\\
U_2=V(u_{23}^2+u_{21}-u_tu_{21}u_{23})\subset\AA_k^4.
\end{gather*}
\end{prop}
\begin{proof}
The quotient variety \(Y\) has a singular locus defined by
\[
\Sing(Y)=V(y_1,y_2,y_4).
\]
Blowing-up \(Y\) along \(\Sing(Y)\), we get
\[
Y^\prime=V(uy_1-ty_2,vy_1-ty_4,vy_2-uy_4, v^2+u^2y_2+t^2y_1y_3-tuy_4)
\],
which is a subvariety of \(\PP_{K[y_1,y_2,y_3,y_4]}^2\), and \(t,u,v\) is a homogeneous coordinate of this projective variety. As
\begin{align*}
Y^\prime\cap V(t,u)&=V(t,u,vy_1,vy_2,v^2)\subset V(v)
\end{align*}
Hence, the variety \(Y^\prime\) is covered by open subvarieties \(V(t)^c,v(u)^c\) of \(\PP_{K[y_1,y_2,y_3,y_4]}^2\). We observe that
\begin{gather*}
V(t)^c=\Spec k[y_1,y_2,y_3,y_4,u/t,v/t],\\
V(u)^c=\Spec k[y_1,y_2,y_3,y_4,t/u,v/u].
\end{gather*}
For simplicity, we denote \(u/t,v/t,t/u,v/u\) by \(t_u,t_v,u_t,u_v\), respectively. Therefore, \(Y^\prime\) is covered by affine open subvarieties
\begin{gather*}
T:=Y^\prime\cap V(t)^c=V(y_2-t_uy_1,y_4-v_ty_1,t_v^2+t_u^3y_1+y_1y_3-t_ut_vy_1),\\
U:=Y^\prime\cap V(u)^c=V(y_1-u_ty_2,y_4-v_uy_2,u_v^2+y_2+u_t^3y_2y_3-u_tu_vy_2).
\end{gather*}
Thus,
\begin{gather*}
T\cong V(t_v^2+t_u^3y_1+y_1y_3-t_ut_vy_1)\subset\AA^4_{y_1,y_3,t_u,t_v}\\
U\cong V(u_v^2+y_2+u_t^3y_2y_3-u_tu_vy_2)\subset\AA^4_{y_2,y_3,u_t,u_v}.
\end{gather*}
Note that \(T,U\) are singular.

We blow up \(T,U\) along these singular loci.
Since the Jacobian matrices \(J_T,J_U\) of \(T,U\) are
\begin{gather*}
J_T=[t_u^3+y_3-t_ut_v,y_1,-t_vy_1,-t_v-t_uy_1],\\
J_U=[1+u_t^3y_3-u_tu_v,u_t^3y_2,-u_vy_2,-u_v-u_ty_2].
\end{gather*}
Thus, the singular loci of \(T\) and \(U\) are given by
\begin{gather*}
\Sing(T)=V(y_1,t_u^3+y_3,t_u)\\
\Sing(U)=V(y_2,1+u_t^3y_3,u_v).
\end{gather*}
First we compute the blow-up \(\widetilde T\) of \(T\) along \(\Sing(T)\). Blowing-up \(T\), we obtain
\[
\widetilde T=V(t_1(t_u^3+y_3)-t_2y_1,t_2t_v-t_3(t_u^3+y_3),t_3y_1-t_1t_v,t_3^2+t_1t_2-t_ut_1t_3),\\
\]
which is a subvariety of \(\PP^2_{k[y_1,y_3,t_u,t_v]}\) and \(t_1,t_2,t_3\) are its homogeneous coordinates. As
\[\widetilde T\cap V(t_1,t_2)=V(t_1,t_2,t_3(t_u^3+y_3),t_3y_1,t_3^2)=\emptyset,\]
the variety \(\widetilde T\) is covered by open subvarieties \(V(t_1)^c,V(t_2)^c\) of \(\PP^2_{k[y_1,y_3,t_u,t_v]}\).
We can write these subvarieties as
\begin{gather*}
V(t_1)^c=\Spec k[y_1,y_3,t_u,t_v,t_2/t_1,t_3/t_1],\\
V(t_2)^c=\Spec k[y_1,y_3,t_u,t_v,t_1/t_2,t_3/t_2].
\end{gather*}
We denote \(t_j/t_i\) by \(t_{ij}\) for convention. Then, we obtain
\begin{gather*}
T_1:=\widetilde T\cap V(t_1)^c=V(t_u^3+y_3-t_{12}y_1,t_v-t_{13}y_1,t_{13}^2+t_{12}-t_ut_{13})\\
T_2:=\widetilde T\cap V(t_2)^c=V(y_1-t_{21}(t_u^3+y_3),t_v-t_{23}(t_v^3+y_3),t_{23}^2+t_{21}-t_ut_{21}t_{23}).
\end{gather*}
These are open cover of \(\widetilde T\). We have
\begin{gather*}
T_1\cong \AA^3_{y_1,t_u,t_{13}},\\
T_2\cong V(t_{23}^2+t_{21}-t_ut_{21}t_{23})\subset\AA^4_{y_3,t_u,t_{21},t_{23}}.
\end{gather*}

Next, we compute the blow-up \(\widetilde U\) of \(U\) along \(\Sing(U)\). That is
\[
\widetilde U=V(u_1(1+u_t^3y_3)-u_2y_2,u_2u_v-u_3(1+u_t^3y_3),u_3y_2-u_1u_v,u_3+u_1u_2-u_tu_1u_3),
\],
which is a subvariety of \(\PP_{K[y_2,y_3,u_t,u_v]}^2\), and \(u_1,u_2,u_3\) are its homogeneous coordinates. As
\[\widetilde U\cap V(u_1,u_2)=V(u_1,u_2,u_3(1+u_t^3y_3),u_3y_2,u_3^2)=\emptyset,\]
the variety \(\widetilde U\) is covered by open affine varieties \(V(u_1)^c, V(u_2)^c\) of \(\PP_{K[y_2,y_3,u_t,u_v]}^2\). We can write
\begin{gather*}
V(u_1)^c=\Spec k[y_2,y_3,u_t,u_v,u_2/u_1,u_3/u_1],\\
V(u_2)^c=\Spec k[y_2,y_3,u_t,u_v,u_1/u_2,u_3/u_2]. 
\end{gather*}
We denote \(u_j/u_i\) by \(u_{ij}\) for convention. Then, we obtain
\begin{gather*}
U_1:=\widetilde U\cap V(u_1)^c=V(1+u_t^3y_3-u_{12}y_2,u_v-u_{13}y_2,u_{13}^2+u_{12}-u_tu_{13}^2),\\
U_2:=\widetilde U\cap V(u_2)^c=V(y_2-u_{23}(1+u_t^3y_3),u_v-u_{23}(1+u_t^3y_3),u_{23}^2+u_{21}-u_tu_{21}u_{23}).
\end{gather*}
Hence, we have
\begin{gather*}
U_1\cong V(1+u_t^3y_3-u_{12}y_2, u_{13}^2+u_{12}-u_tu_{13}^2)\subset\AA_{y_2,y_3,u_t,u_{12},u_{13}}^5,\\
U_2\cong V(u_{23}^2+u_{21}-u_tu_{21}u_{23})\subset\AA_{y_3,u_t,u_{21},u_{23}}^4.
\end{gather*}

We obtain the blow-up \(\widetilde Y\) of \(Y^\prime\) along its singular locus by gluing \(\widetilde U, \widetilde T\), or equivalently, gluing \(T_1,T_2,U_1,U_2\). As
\[U_1\cap V(u_t,u_{12})=V(u_t,u_{12},1,u_{13}^2)=\emptyset,\]
the subvariety \(U_1\) is contained in \(U_2\cup\widetilde T\). Hence \(\widetilde Y\) is given by gluing \(T_1,T_2,U_2\). Evidently, \(T_1\) is nonsingular. The Jacobian matrices of \(T_2,U_2\) are
\begin{gather*}
J_{T_1}=[0,-t_{13},1,-t_{13}-t_u]
J_{T_2}=[0,-t_{21}t_{23},1-t_ut_{23},-t_{23}-t_ut_{21}],\\
J_{U_2}=[0,-u_{21}u_{23},1-u_tu_{23},-u_{23}-u_tu_{21}].
\end{gather*}
Hence, \(T_2, U_2\) are nonsingular, and so is \(\widetilde Y\).

We can observe that this resolution of \(Y\) is crepant from the following proposition and corollary.
\end{proof}

\begin{prop}
Let \(V\) be a quasi-projective integral variety, and \(W\) be a hypersurface of \(V\). We denote the blow-up of \(V\) and \(W\) along \(C\subset W\) by \(\widetilde V\) and \(\widetilde W\), respectively. We abuse \(f\) to denote morphisms \(\widetilde V\to V\) and \(\widetilde W\to W\). If \(C\) is a smooth irreducible subset of \(W\), then
\[
K_{\widetilde W}=f^\ast K_W+(r-m-1)E|_{\widetilde W}
\]
where \(r\) is codimension of \(C\) in \(V\), \(m\) is multiplicity of \(W\) along \(C\), and \(E\) is the exceptional divisor of \(f\colon \widetilde V\to V\).
\end{prop}
\begin{proof}
As \(\widetilde V\) is blow-up of \(V\), we have
\[K_{\widetilde V}=f^\ast K_V+(r-1)E\]
and
\[f^\ast S=\widetilde S+mE.\]
Combining these, we obtain
\[K_{\widetilde V}+\widetilde S=f^\ast(K_V+S)+(r-m-1)E.\]
By restricting on \(\widetilde S\) and applying the adjunction formula, we obtain
\[K_{\widetilde S}=f^\ast K_S+(r-m-1)E|_{\widetilde S}\]
that is the required formula.
\end{proof}
\begin{coro}
The resolution of \(Y\) given in Proposition \ref{3-prop:resolY} is crepant.
\end{coro}
\begin{proof}
We use the same notation as in proof of Proposition \ref{3-prop:resolY}. The resolution is given by twice the blow-up. 

The first blow-up \(Y^\prime\to Y\) is that of \(Y\) along \(\Sing(Y)\). The variety \(Y\) is hypersurface of \(\AA^4\) and \(\Sing(Y)\) is a line in \(\AA^4\). As multiplicity \(Y\) along \(\Sing(Y)\) is two, \(Y^\prime\to Y\) is crepant by the previous proposition. 

The second blow-up \(\widetilde Y\to Y^\prime\) is given by gluing two blow-ups \(\widetilde T\to T\) and \(\widetilde U\to U\). We can regard \(T,U\) as hypersurfaces of \(\AA^4\). Then its singular loci are curves on \(\AA^4\), and the multiplicities of \(T,U\) on its singular locus is two. Thus, the blow-ups \(\widetilde T\to T\) and \(\widetilde U\to U\) are crepant from previous proposition, and hence \(\widetilde Y\to Y^\prime\) is crepant.

As a composition of crepant morphisms is also crepant, the resolution \(\widetilde Y\to Y\) is crepant.
\end{proof}

To compute the resolution of \(X\), we describe the \(\bsym T\)-action on \(\widetilde Y\) given as the extension of the \(\bsym T\)-action on \(Y\)

\begin{prop}
The extension of the \(\bsym T\)-action on \(Y\) to \(\widetilde Y\) is given by the \(\bsym T\)-action on \(T_1,T_2,U_2\) defined by the following actions on those coordinate rings:
\begin{gather*}
[y_1,t_u,t_{13}]\cdot\bsym T=[-y_1,-t_u,t_{13}-t_u],\\
[y_3,t_u,t_{21},t_{23}]\cdot\bsym T=[-y_3,-t_u,t_{21},t_{23}-t_{21}t_u],\\
[y_3,u_t,u_{21},u_{23}]\cdot\bsym T=[-y_3,-u_t,u_{21},u_{23}-u_{21}u_t].
\end{gather*}
\end{prop}
\begin{proof}
We consider the following diagram
\[\xymatrix{
\widetilde Y\ar[r]^{\bsym T}\ar[d]&\widetilde Y\ar[d]\\
Y\ar[r]_{\bsym T}&Y
}\]
The \(\bsym T\)-action on \(\widetilde Y\) is characterized as the action that makes the above diagram commutative. On \(T_1\), the following diagram is commutative.
\[\xymatrix{
k[y_1,t_u,t_{13}]\ar[r]^{\bsym T}&k[y_1,t_u,t_{13}]\\
k[y_1,y_2,y_3,y_4]/I_Y\ar[r]_{\bsym T}\ar[u]&k[y_1,y_2,y_3,y_4]/I_Y\ar[u]
}\]
Note that \(I_Y=(y_4^2+y_2^3+y_1^3y_3-y_1y_2y_4)\) is a defining ideal of \(Y\) in \(\AA^4\), and the \(\bsym T\)-action on the coordinate ring of \(Y\) is defined by (\ref{eq:actOnY}). Hence, We obtain
\begin{gather*}
\begin{aligned}
y_1\cdot\bsym T&=\varphi(y_1)\cdot\bsym T\\
&=\varphi(y_1\cdot\bsym T)=-y_1,
\end{aligned}\\
\begin{aligned}
(t_uy_1)\cdot\bsym T&=\varphi(y_2)\cdot\bsym T\\
&=\varphi(y_2\cdot\bsym T)=t_uy_1,
\end{aligned}\\
\begin{aligned}
(t_{13}(t_u-t_{13})y_1-t_u^3)\cdot T&=\varphi(y_3)\cdot\bsym T\\
&=\varphi(y_3\cdot\bsym T)=-t_{13}(t_u-t_{13})y_1+t_u^3,
\end{aligned}\\
\begin{aligned}
(t_{13}y_1^2)\cdot\bsym T&=\varphi(y_4)\cdot\bsym T\\
&=\varphi(y_4\cdot T)=t_{13}y_1^2-t_uy_1^2.
\end{aligned}
\end{gather*}
Combining these, we obtain
\[
[y_1,t_u,t_{13}]\cdot\bsym T=[-y_1,-t_u,t_{13}-t_u].\\
\]
By similar computation, we obtain the \(\bsym T\)-actions on the coordinate rings of \(T_2,U_2\) as
\begin{gather*}
[y_3,t_u,t_{21},t_{23}]\cdot\bsym T=[-y_3,-t_u,t_{21},t_{23}-t_{21}t_u],\\
[y_3,u_t,u_{21},u_{23}]\cdot\bsym T=[-y_3,-u_t,u_{21},u_{23}-u_{21}u_t].
\end{gather*}
\end{proof}

We give a crepant resolution of the quotient variety \(X=\AA^3_k/G\) and compute its \(l\)-adic Euler characteristic. For a prime \(l\) that is different from the characteristic \(p\), the \(l\)-adic Euler characteristic is defined by the alternating sum of the dimension of the \(l\)-adic \'etale cohomology \(H_{et,c}^i(X)\) with compact support:
\[\chi(X)=\sum_{i=0}^{2\dim X}(-1)^i\dim H_{et,c}^i(X)\]
\begin{thm}\label{thm:subthm}
Let \(G\subset\SL_k(3)\) be a small subgroup isomorphic to \(\symm_3\). Then the quotient variety \(X=\AA_k^3/G\) has a crepant resolution \(\pi:\widetilde X\to X\). Moreover, 
\[
\chi(\widetilde X)=6.
\]
\end{thm}
\begin{proof}
To show the existence of a crepant resolution of \(X\), we have to only show that there exists a crepant resolution of \(\widetilde Y/\angles{\bsym T}\) constructed above. The variety \(\widetilde Y\) is given by gluing the open affine subvarieties \(T_1, T_2, U_2\). We have
\begin{gather*}
T_1/\angles{\bsym T}\cong\Spec k[z_{11},z_{12},z_{13},z_{14}]/(z_{12}^2-z_{11}z_{13}),\\
T_2/\angles{\bsym T}\cong\Spec k[z_{21},z_{22},\ldots,z_{25}]/(z_{22}^2-z_{21}z_{23},z_{25}^2-z_{23}z_{24}^2+z_{24}),\\
U_2/\angles{\bsym T}\cong\Spec k[z_{31},z_{32},\ldots,z_{35}]/(z_{32}^2-z_{31}z_{33},z_{35}^2-z_{33}z_{34}^2+z_{34}).
\end{gather*}
We can regard these varieties as an open affine cover of \(\widetilde Y/\angles{\bsym T}\). The singular locus in these varieties are
\begin{gather*}
\Sing (T_1/\angles{\bsym T})=V(z_{11},z_{12},z_{13}),\\
\Sing (T_2/\angles{\bsym T})=V(z_{21},z_{22},z_{23},z_{25}^2+z_{24}),\\
\Sing (U_2/\angles{\bsym T})=V(z_{31},z_{32},z_{33},z_{35}^2+z_{34}).
\end{gather*}
Hence, the fixed locus \(F\) of \(\widetilde Y\) has pure dimension one.

As \(\widetilde Y\) is nonsingular, the completion \(\hat{\mathcal O}_{\widetilde Y,y}\) of the local ring \(\mathcal O_{\widetilde Y,y}\) is isomorphic to the ring \(k[[x,y,z]]\) of formal power series for any closed point \(y\in\widetilde Y\). Thus \(\hat{\mathcal O}_{\widetilde Y/\angles{\bsym T},\bar y}\) is isomorphic to \(\hat{\mathcal O}_{Y,y}^{\mathrm{Stab(y)}}\), where \(\bar y\) is image of \(y\) about \(\widetilde Y\to\widetilde Y/\angles{\bsym T}\). If \(y\in F\), then \(\mathrm{Stab}(y)=\angles{\bsym T}\). As \(\angles{\bsym T}\) is tame, the \(\angles{\bsym T}\)-action on \(\hat{\mathcal O}_{\widetilde Y,y}\) can be linearized. As \(F\) has pure dimension one, we may assume that the action is defined by \(\diag(1,-1,-1)\). Let \(Z\) be a fiber product of \(\AA^2_k/(\ZZ/2\ZZ)\) and \(\AA^1_k\) where \(\ZZ/2\ZZ\) acts on \(\AA^2_k\) by \(-I\). Then, we obtain
\[
\hat{\mathcal O}_{\widetilde Y/\angles{\bsym T},\bar y}\cong \hat{\mathcal O}_{Z,o}\cong k[[x,y,z,w]]/(xy-z^2).
\]
The quotient \(\AA_k^2/(\ZZ/2\ZZ)\) has a crepant resolution by blowing-up at origin. Hence, we can obtain a crepant resolution \(\psi\colon\widetilde X\to\widetilde Y/\angles{\bsym T}\) by blowing-up along \(\Sing(\widetilde Y/\angles{\bsym T})\).

To compute the Euler characteristic of \(\widetilde X\), we describe the structure of \(\widetilde X\) explicitly. 
Blowing up the varieties \(T_1,T_2,U_2\) along its singular locus, we can obtain their resolutions. We denote blow-ups of \(T_1/\angles{\bsym T}, T_2/\angles{\bsym T},U_2/\angles{\bsym T}\) by \(\widetilde T_1,\widetilde T_2,\widetilde U_2\) respectively. We describe these as
\begin{gather*}
    \widetilde T_1=V(\{\zeta_iz_{1j}-\zeta_jz_{1i}\mid 1\leq i<j\leq 3\},\zeta_2^2-\zeta_1\zeta_3),\\
    \widetilde T_2=V(\{\eta_iz_{2j}-\eta_jz_{2i}\mid 1\leq i<j\leq 3\},\eta_2^2-\eta_1\eta_3,z_{25}^2-z_{23}z_{24}^2+z_{24}),\\
    \widetilde U_2=V(\{\xi_iz_{3j}-\xi_jz_{3i}\mid 1\leq i<j\leq 3\},\xi_2^2-\xi_1\xi_3,z_{35}^2-z_{33}z_{34}^2+z_{34}),
\end{gather*}
where \(\widetilde T_1\) is embedded in \(\AA_k^4\times_k\PP_k^2\), and the others are embedded in \(\AA_k^5\times_k\PP_k^2\). The letters \(z_{ij}\) are the coordinates of affine spaces \(\AA_k^4\) or \(\AA_k^5\) and \(\zeta_i,\eta_i,\iota_i\) are homogeneous coordinates of \(\PP_k^2\). The varieties \(\widetilde T_1,\widetilde T_2,\widetilde U_2\) are nonsingular varieties. We obtain a nonsingular variety \(\widetilde X\) by gluing \(\widetilde T_1,\widetilde T_2,\widetilde U_2\).

Let \(E\) be the exceptional set of \(\psi\). As
\[\widetilde X\backslash E\cong (\widetilde Y/\angles{\bsym T})\backslash\Sing(\widetilde Y/\angles{\bsym T})\cong (\widetilde Y\backslash F)/\angles{\bsym T},\]
we obtain
\begin{align*}
\chi(\widetilde X\backslash E)&=\chi((\widetilde Y\backslash F)/\angles{\bsym T})\\
&=\frac 12(\chi(\widetilde Y)-\chi(F)).
\end{align*}
As \(\chi(Y)=3\) and \(F\cong\AA^1_k\sqcup\PP^1_k\), we obtain
\[
\chi(\widetilde X\backslash E)=\frac 12(3-(1+2))=0.
\]
We have
\begin{gather*}
E\cap\widetilde T_1=V(z_{11},z_{12},z_{13},\zeta_2^2-\zeta_1\zeta_3)\cong\AA^1_k\times_k\PP_k^1,\\
E\cap(\widetilde T_2\backslash \widetilde T_1)=V(z_{21},z_{22},z_{23},z_{24},z_{25},\eta^2_2-\eta_1\eta_3)\cong\PP_k^1,\\
E\cap\widetilde U_2=V(z_{31},z_{32},z_{33},z_{35}^2+z_{34},\xi_2^2-\xi_1\xi_3)\cong\AA^1_k\times_k\PP^1_k.
\end{gather*}
As \(E\cap(\widetilde T_1\cup\widetilde T_2)\) and \(E\cap\widetilde U_2\) are disjoint, we obtain 
\[
\chi(E)=2+2+2=6.
\]
Therefore,
\[
\chi(X)=\chi(X\backslash E)+\chi(E)=6.
\]
\end{proof}
\section{Quotient singularity for general}
In this section, we demonstrate the main theorem.
\begin{thm}\label{main}
Let \(G\) be a small finite subgroup of \(\SL_3(k)\) which is a semi-product \(H\rtimes G^\prime\) of a tame Abelian group \(H\) and a cyclic group \(G^\prime\cong\ZZ/3\ZZ\) or the symmetric group \(G^\prime\cong\symm_3\). If \(G\) acts on the affine space \(\AA^3_k\) canonically, then the quotient variety \(X=\AA^3/G\) has a crepant resolution \(\widetilde X\to X\). Moreover, we have the following equalities:
\[
\chi(\widetilde X)=\begin{cases}
\sharp\mathrm{Conj}(G)&(G^\prime\cong\ZZ/3\ZZ)\\
\sharp\mathrm{Conj}(G)+3&(G^\prime\cong\symm_3)
\end{cases}
\]
\end{thm}

We prove this theorem by constructing a crepant resolution along the following strategy. We set \(Y=\AA^3/H\). We devise the following diagram:
\[\xymatrix{
&&\widetilde X\ar[d]\\
&\widetilde Y\ar[r]\ar[d]&X^\prime\ar[d]\\
\AA^3\ar[r]&Y\ar[r]&X
}\]
As \(H\) is a tame Abelian, we can construct \(Y\) as a toric variety. We formulate a toric crepant resolution \(\widetilde Y\to Y\) such that the \(G^\prime\)-action on \(Y\) extends on \(\widetilde Y\). We set \(X^\prime=\widetilde Y/G^\prime\). We show that \(X^\prime\) has a crepant resolution \(\widetilde X\to X^\prime\). The composition \(\widetilde X\to X^\prime\to X\) is the desired crepant resolution of \(X\).

First, we illustrate the construction of \(Y\) as a toric variety. Let \(r\) be the maximal order of the elements of \(H\). Fix a \(r\)-th primitive root \(\zeta_r\in k\). As all element of \(H\) are diagonal matrices, and its order is a divisor of \(r\), any \(\bsym X\in H\backslash\{I\}\) is written as \(\diag(\zeta_r^a,\zeta_r^b,\zeta_r^c)\) for some integers \(a,b,c\) satisfying \(0\leq a,b,c<r\). We define a map
\[pt\colon H\ni\diag(\zeta_r^a,\zeta_r^b,\zeta_r^c)\to\frac 1r[a,b,c]\in\RR.\]
Let \(\Gamma\) be the lattice generated by \(\ZZ^3\) and \(pt(H)\). Then, the quotient variety \(Y\) arises as the toric variety defined by the lattice \(\Gamma\) and the cone \(\RR^3_{\geq 0}\). 

We define \(Tr\colon\RR^3\to\RR\) by the sum of its entries. aS \(H\in\SL_3(k)\), \(Tr(pt(H))\subset\ZZ\). Hence, \(Tr(\Gamma)\subset\ZZ\).
Let \(\Delta_1\) be the plane in \(\RR^3\) defined by \(x+y+z=1\). The following lemma gives a way to construct a crepant resolution of \(Y\).
\begin{lemm}\label{lem:cirtSing}
Let \(\bsym p_1,\bsym p_2,\bsym p_3\) be the points in \(\Gamma\cap\Delta_1\). Suppose no points of \(\Gamma\) are in the triangle \(\bsym p_1,\bsym p_2,\bsym p_3\) except its vertices. Then \(\bsym p_1,\bsym p_2,\bsym p_3\) is a basis of \(\Gamma\).
\end{lemm}
\begin{proof}
As \(\bsym p_1,\bsym p_2,\bsym p_3\) is a basis of \(\RR^3\) as a \(\RR\)-vector space, we only have to show that \(\bsym p_1,\bsym p_2,\bsym p_3\) generate \(\Gamma\). Let \(\Gamma_0\) be the lattice generated by \(\bsym p_1,\bsym p_2,\bsym p_3\). To show \(\Gamma=\Gamma_0\), we suppose that \(\Gamma\ne\Gamma_0\), on the contrary. Evidently, \(\Gamma_0\subset\Gamma\). Hence, there exists \(\bsym x=[x_1,x_2,x_3]\in\Gamma\backslash\Gamma_0\). As \(\bsym p_1,\bsym p_2,\bsym p_3\) is a \(\RR\)-basis of \(\RR^3\), we can write
\[
\bsym x=\sum_{i=1}^3c_i\bsym p_i
\]
by some \(c_i\in\RR\). As
\[
\bsym x-\sum_{i=1}^3\floor{c_i}\bsym p_i=\sum_{i=1}^3(c_i-\floor{c_i})\bsym p_i\in\Gamma\backslash\Gamma_0,
\]
we may assume that \(0\leq c_i<1\) for \(i=1,2,3\). As \(\bsym p_1,\bsym p_2,\bsym p_3\in\Delta_1\) and \(Tr\) is \(\RR\)-linear, we obtain
\[
\sum_{i=1}^3c_i=\sum_{i=1}^3c_iTr(\bsym p_i)=Tr(\bsym x)\in\ZZ.
\]
Thus
\[
\sum_{i=1}^3c_i=1\text{ or }2.
\]
Replacing \(\bsym x\) by
\[
\sum_{i=1}^3\bsym p_i-\bsym x=\sum_{i=1}^3(1-c_i)\bsym p_i\in\Gamma\backslash\Gamma_0
\]
if necessary, we may assume that \(0\leq c_i\leq 1\) and \(c_1+c_2+c_3=1\). This implies \(\bsym x\) is a point of \(\Gamma\) in the triangle \(\bsym p_1\bsym p_2\bsym p_3\). By assumption, \(\bsym x\) is one of the vertices. It contradicts that \(\bsym x\not\in\Gamma_0\).
\end{proof}

Lemma \ref{lem:cirtSing} implies that to construct a crepant resolution of \(Y\), we have to only provide a subdivision of the triangle \(T=\RR^3_{\geq 0}\cap\Delta_1\) into triangles whose vertices are exactly all the points of \(\Gamma\cap T\).

\subsection{The case \(G^\prime\cong\ZZ/3\ZZ\)}
First, we consider the case that \(G\) is a semidirect product of a tame Abelian group \(H\) and a cyclic group \(\ZZ/3\ZZ\). We may assume that each element of \(H\) is diagonal and \(G\) generated by \(H\) and
\[
\bsym S=\begin{bmatrix}
0&1&0\\
0&0&1\\
1&0&0
\end{bmatrix}
\]
from Lemma \ref{lem:gr1}. Thus \(G^\prime=\angles{\bsym S}\). The \(\bsym S\)-action on \(Y\) is given by the toric automorphism corresponding to the canonical \(\bsym S\)-action on \(\RR^3\). 
\begin{figure}
    \centering
    \includegraphics[width=5cm,pagebox=cropbox]{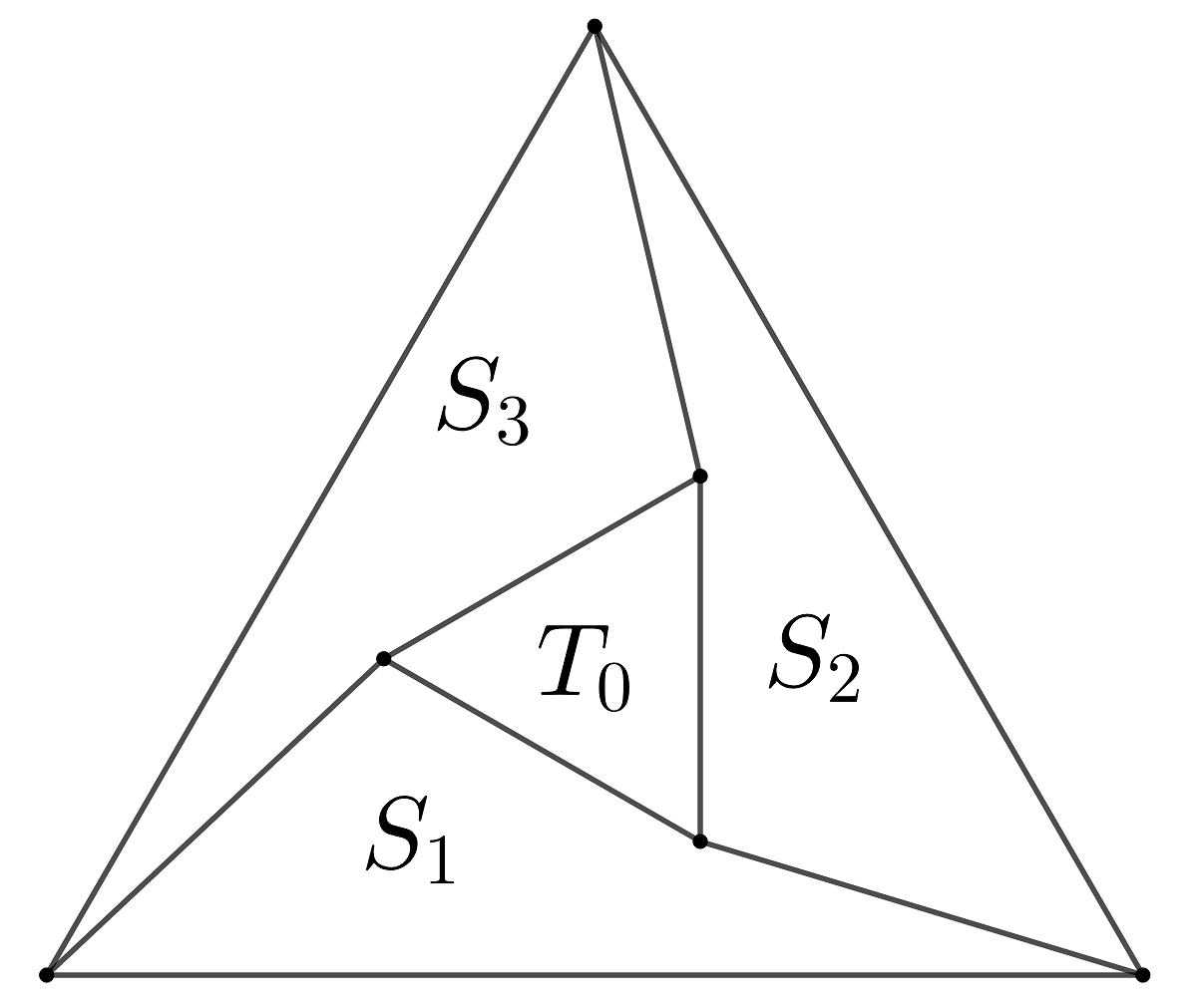}
    \caption{subdivision of \(T\)}
    \label{fig:my_label}
\end{figure}
To define a crepant resolution of \(Y\), we provide a subdivision of \(T\). Take a point \(\bsym a\in H\cap\Delta_1\) such that its distance from the center \([\frac 13,\frac 13,\frac 13]\) of \(T\) is minimal among all points in \(\Gamma\cap\Delta_1\). Note that \([\frac 13,\frac 13,\frac 13]\not\in\Gamma\). Let \(\bsym a^\prime=\bsym a\cdot\bsym S,\ \bsym a^\dprime=\bsym a\cdot\bsym S^2\), and let \(T_0\) be the triangle \(\bsym a\bsym a^\prime\bsym a^\dprime\). Evidently, \(T_0\) is stable for the \(\bsym S\)-action. We may assume that \(\bsym a\) in \(\{[x,y,z]\in\RR^3\mid x\leq y,x\leq z\}\). We denote the square \(\bsym a\bsym a^\prime\bsym e_x\bsym e_z\) by \(S_1\) where \(\bsym e_x=[1,0,0], \bsym e_z=[0,0,1]\). By putting \(S_2=S_1\cdot\bsym S,\ S_3=S_1\cdot\bsym S^2\), the polygons \(T_0,S_1,S_2,S_3\) present a subdivision of \(T\) (Figure \ref{fig:my_label}). Subdivide \(S_1\) into triangles whose vertices are exactly all the points of \(\Gamma\cap S_1\). The subdivision of \(S_1\) persents a subdivisions of \(S_2,S_3\) by the action of \(\bsym S\). Thus, we obtain a \(\bsym S\)-stable subdivision of \(T\) into triangles whose vertices are exactly \(\Gamma\cap T\).

We denote the fan given by the above subdivision of \(T\) by \(\Sigma\) and the toric variety corresponding to \(\Sigma\) by \(\widetilde Y\). From Lemma \ref{lem:cirtSing}, \(\widetilde Y\) is nonsingular. As all the rays of \(\Sigma\) have the primitive point in \(\Delta_1\), the resolution \(\widetilde Y\to Y\) is crepant. Moreover, the \(\bsym S\)-action on \(Y\) lifts on \(\widetilde Y\) because \(\Sigma\) is \(\bsym S\)-stable.

Next we present a crepant resolution of \(\widetilde Y/\bsym S\). We denote the orbit of the torus action on \(\widetilde Y\) corresponding to a cone \(\sigma\in\Sigma\) by \(O(\sigma)\). Then,
\[
\widetilde Y/\bsym S=O(\bsym o)/\bsym S\sqcup O(\tau)/\bsym S\sqcup\left(\bigsqcup_{\sigma\in (\Sigma-\{\bsym o,\tau\})/\bsym S}O(\sigma)\right)
\]
Where, \(\bsym o\) is the origin of \(\RR^3\), \(\tau\) is the 3-dimensional cone corresponding to \(T_0\), and \((\Sigma-\{\bsym o,\tau\})/\bsym S\) is the set of the representatives of the \(\bsym S\)-orbits in \(\Sigma-\{\bsym o,\tau\}\). Hence, the singular points in \(\widetilde Y/\bsym S\) are contained in \(O(\bsym o)/\bsym S\sqcup O(\tau)/\bsym S\), or the quotient \(U/\bsym S\) of the open affine subvariety \(U\cong\AA^3\) of \(Y\) corresponding the cone \(\tau\). From Proposition \ref{3-prop:resolY}, \(U/\bsym S\) has a crepant resolution. Therefore \(\widetilde Y/\bsym S\) has a crepant resolution \(\widetilde X\to\widetilde Y/\bsym S\).

Using this construction, we prove a part of the main theorem.
\begin{thm}
Let \(G\) be a small subgroup of \(\SL_3(k)\) isomorphic to a semi-product \(H\rtimes\ZZ/3\ZZ\) where \(H\) is a tame Abelian group. When \(G\) acts on \(\AA^3_k\) canonically, the quotient variety \(X=\AA^3_k/G\) has a crepant resolution \(\widetilde X\to X\), and the Euler characteristic \(\chi(\widetilde X)\) is equal to the number of conjugacy classes \(\mathrm{Conj}(G)\) of \(G\). 
\end{thm}
\begin{proof}
A crepant resolution of \(X\) is given above. Hence we have only to prove the equality of the Euler characteristic and the number of conjugacy classes. As \(H\) is tame and Abelian, the Euler characteristic \(\chi(\widetilde Y)\) of \(\widetilde Y\) is \(\sharp H\). Using the decomposition of \(\widetilde Y\) into torus orbits, we obtain
\[
[\widetilde Y]=\sum_{\sigma\in\Sigma}[O(\sigma)]
\]
in the Grothendiek ring of varieties. As the Euler characteristic provides an additive map from the Grothendiek ring to \(\ZZ\), we obtain
\[
\sum_{\sigma\in\Sigma}\chi(O(\sigma))=\chi(\widetilde Y)=\sharp H.
\]
If \(\dim(\sigma)=d\), then \(O(\sigma)\cong\mathbb T^{3-d}\) where \(\mathbb T^n\) is the \(n\)-dimensional torus, and hence we obtain 
\[\chi(O(\sigma))=\begin{cases}
0&(\dim(\sigma)\leq 2)\\
1&(\dim(\sigma)=3)
\end{cases}\]
Let \(\Sigma_3\) be the set of cones of dimension three in \(\Sigma\). we obtain
\[
\sharp\Sigma_3=\sum_{\sigma\in\Sigma_3}\chi(O(\sigma))=\sharp H.
\]

Meanwhile, \([\widetilde X]\) is decomposed as
\[
[\widetilde X]=\sum_{\sigma\in(\Sigma-|\tau|)/\bsym S}[O(\sigma)]+[\widetilde{\AA^3/\bsym S}]
\]
where \(|\tau|\) is the set of the faces of cone \(\tau\), \((\Sigma-|\tau|)/\bsym S\) is the set of the representatives of the \(\bsym S\)-orbits on \(\Sigma-|\tau|\), and \(\widetilde{\AA^3/\bsym S}\) is a crepant resolution of \(\AA^3/\bsym S\). Then, we have
\begin{align*}
\chi(\widetilde X)&=\sum_{\sigma\in(\Sigma_3-\{\tau\})/\bsym S}\chi(O(\sigma))+\chi(\widetilde{\AA^3/\bsym S})\\
&=\sharp((\Sigma_3-\{\tau\})/\bsym S)+\chi(\widetilde{\AA^3/\bsym S})\\
&=\frac{\sharp H-1}3+\chi(\widetilde{\AA^3/\bsym S}).
\end{align*}
By \cite[Corollary 6.21]{Yas}, \(\chi(\widetilde{\AA^3/\bsym S})=3\). Therefore,
\[
\chi(\widetilde X)=\frac{\sharp H-1}3+3=\sharp\mathrm{Conj}(G)
\]
from Lemma \ref{lem:conjnum1}.
\end{proof}

\subsection{The case \(G^\prime\cong\symm_3\)}
We consider the case that \(G\) is a semi-product of a tame Abelian group \(H\) and a symmetric group \(\symm_3\). We may assume that each element of \(H\) is diagonal and \(G\) is generated by \(H\),
\[
\bsym S=\begin{bmatrix}
0&1&0\\
0&0&1\\
1&0&0
\end{bmatrix},\text{ and }
\bsym T=\begin{bmatrix}
0&0&-1\\
0&-1&0\\
-1&0&0
\end{bmatrix}
\]
from Lemma \ref{lem:gr2}. Thus \(G^\prime=\angles{\bsym S,\bsym T}\). We define a \(G^\prime\)-action on \(\RR^3\) by
\begin{gather*}
[x,y,z]\cdot\bsym S=[z,x,y],\\
[x,y,z]\cdot\bsym T=[z,y,x].
\end{gather*}
Note that the \(\bsym T\)-action on \(\AA^3\) corresponds to the composition of the toric automorphism defined by the \(\bsym T\)-action on \(\RR^3\) and the action of \([-1,-1,-1]\in\bsym T^3\).

\begin{figure}
    \centering
    \includegraphics[width=5cm,pagebox=cropbox]{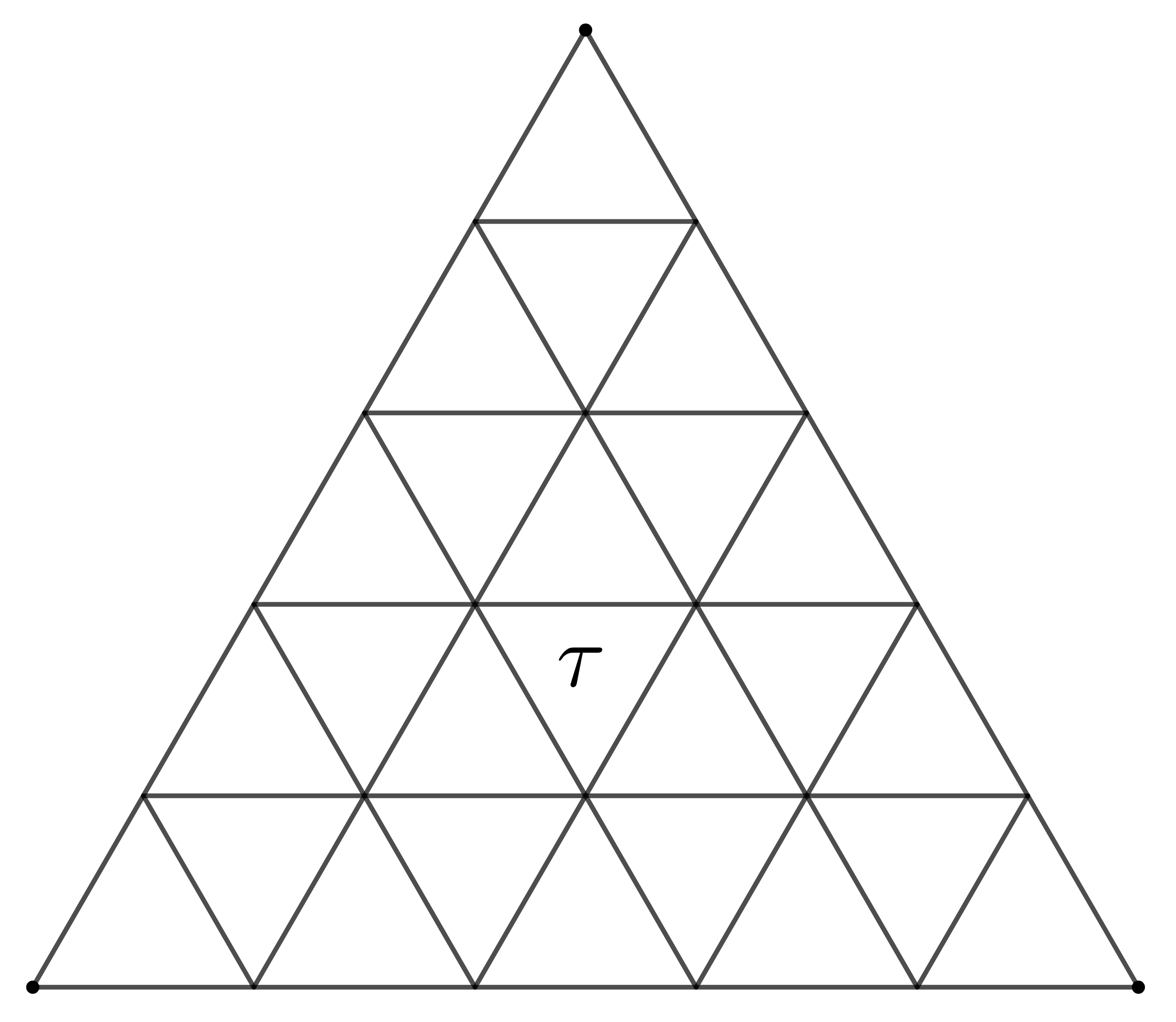}
    \caption{subdivision of \(T\) where \(r=5\)}
    \label{fig:my_label2}
\end{figure}

We provide a \(G^\prime\)-stable subdivision of \(T\). From Lemma \ref{lem:aboutH}, we obtain
\[
\bsym T\cap\Gamma=\left\{\frac 1r[x,y,z]\middle|x,y,z\in\ZZ,\ 0\leq x,y,z\leq r,x+y+z=r\right\}.
\]
Hence, the subdivision of \(T\) with \(3(r-1)\) lines
\[
x=\frac ir,\ y=\frac ir,\ z=\frac ir\quad(1\leq i<r)
\]
in \(\Delta_1\) consists triangles whose vertices are exactly the points of \(\Gamma\cap T\) (Figure \ref{fig:my_label2}). Evidently, the subdivision is \(G^\prime\)-stable. Let \(\Sigma\) be the fan corresponding to the subdivision of \(T\), and let \(\widetilde Y\) be the toric variety defined by the lattice \(\Gamma\) and the fan \(\Sigma\). Then \(\widetilde Y\to Y\) is a crepant resolution of \(Y\) as in section 4.1. 

We construct a crepant resolution of \(\widetilde Y/G^\prime\). Now \(G^\prime\) is the group \(\angles{\bsym S,\bsym T}\). We can construct a crepant resolution \(Y^\prime\to\widetilde Y/\bsym S\) as in section 4.1. Hence, we have to construct a crepant resolution of \(Y^\prime/\bsym T\). By construction, \(Y^\prime\) has the form
\[
Y^\prime=\bigsqcup_{\sigma\in(\Sigma-|\tau|)/\bsym S}O(\sigma)\sqcup\widetilde{\AA^3/\bsym S}
\]
Where, we use the same notation as in the proof of Theorem 4.3. Let \(\Sigma^\prime\) be the set of the cones whose \(G^\prime\)-orbit has exactly three elements. Then
\[
Y^\prime=\bigsqcup_{\sigma\in(\Sigma-\Sigma^\prime)/G^\prime}O(\sigma)\sqcup\bigsqcup_{\sigma^\prime\in(\Sigma^\prime-|\tau|)/\bsym S}O(\sigma^\prime)/\bsym T\sqcup(\widetilde{\AA^3/\bsym S})/\bsym T
\]
where \((\Sigma^\prime-|\tau|)/\bsym S\) is the set of representatives which stable with respect to the action of \(\bsym T\).
Any cone \(\sigma^\prime\) in \((\Sigma^\prime-|\tau|)/\bsym S\) is a face of a 3-dimensional cone \(c\in(\Sigma^\prime-|\tau|)/\bsym S\). In this case, \(O(\sigma^\prime)\subset U_c\) where \(U_c\) is the open affine subvariety of \(Y\) respect to \(c\). If we write \(U_c=\Spec k[s,t,u]\), the \(\bsym T\)-action on \(U_c\) corresponds to the ring homomorphism defined by
\[s\mapsto -t,\ t\mapsto -s,\ u\mapsto -u.\]
Hence, the singular locus in \(U_c/\bsym T\) is 1-dimensional, and hence the one of \(Y^\prime/\bsym T\) is so. Therefore, we can obtain a crepant resolution of \(Y^\prime/\bsym T\) by blow-up along its singular locus as in the proof of Theorem \ref{thm:subthm}.

From the argumentabove, we get the remaining part of the main theorem.
\begin{thm}\label{thm:main2}
Let \(G\) be a subgroup of \(\SL_3(k)\) isomorphic to \(H\rtimes\symm_3\) where \(H\) is a tame Abelian group. Suppose \(G\) acts on \(\AA^3_k\) canonically. Then the quotient variety \(X=\AA^3/G\) has a crepant resolution \(\widetilde X\to X\). Moreover, the following equality holds.
\[
\chi(\widetilde X)=\sharp\mathrm{Conj}(G)+3
\]
\end{thm}
\begin{proof}
Let \(\widetilde X\to X\) be the crepant resolution constructed above. Then \(\widetilde X\) has the form
\[
\widetilde X=\bigsqcup_{\sigma\in(\Sigma-\Sigma^\prime)/G^\prime}O(\sigma)\sqcup\bigsqcup_{\sigma^\prime\in(\Sigma^\prime-|\tau|)/\bsym S}\widetilde{O(\sigma^\prime)/\bsym T}\sqcup\widetilde{\AA^3/G^\prime}
\]
Where, \(\widetilde{O(\sigma^\prime)/\bsym T}\) is blow-up of \(O(\sigma^\prime)/\bsym T\) along singular locus of \(Y^\prime\), and \(\widetilde{\AA^3/G^\prime}\) is the crepant resolution of \(\AA^3/G^\prime\) constructed in Theorem \ref{thm:subthm}.
Hence, the Euler characteristic of \(\widetilde X\) is computed as
\[
\chi(\widetilde X)=\sum_{\sigma\in(\Sigma-\Sigma^\prime)/G^\prime}\chi(O(\sigma))+\sum_{\sigma^\prime\in(\Sigma^\prime-|\tau|)/\bsym S}\chi(\widetilde{O(\sigma^\prime)/\bsym T})+\chi(\widetilde{\AA^3/G^\prime}).
\]

We denote the set of 3-dimensional cones in \(\Sigma,\Sigma^\prime\) by \(\Sigma_3,\Sigma^\prime_3\) respectively. From the construction of the fan \(\Sigma\), we have
\[\sharp\Sigma^\prime_3=3r-2,\ \sharp\Sigma_3=r^2.\]
Then we have
\begin{align*}
\sum_{\sigma\in(\Sigma-\Sigma^\prime)/G^\prime}\chi(O(\sigma))&=\sharp(\Sigma_3-\Sigma^\prime_3)/6\\
&=\frac{r^2-3r+2}6=\frac{(r-1)(r-2)}6.
\end{align*}

We compute \(\chi(\widetilde{O(\sigma^\prime)/\bsym T})\) for \(\sigma^\prime\in\Sigma^\prime-|\tau|\). The cone \(\sigma^\prime\) is a face of \(c\in\Sigma^\prime_3\). When we write the affine open subspace \(U_c\) of \(Y^\prime\) corresponding to \(c\) as \(U_c=\Spec k[s,t,u]\), the action of \(\bsym T\) on \(U_c\) is defined by
\[
s\mapsto -t,\ t\mapsto -s,\ u\mapsto -u
\]

If \(\dim(\sigma^\prime)=3\), then \(O(\sigma^\prime)\) is the origin of \(U_c\) and it is a point on the singular locus of \(Y^\prime\). Hence \(\widetilde{O(\sigma^\prime)/\bsym T}\) is \(\PP^1\). Therefore,
\[
\chi(\widetilde{O(\sigma^\prime)/\bsym T})=2.
\]

If \(\dim(\sigma^\prime)=2\), then \(O(\sigma^\prime)\) is the hyperplane \(V(u)\) of \(U_c\). Then \(\bsym T\) acts on \(O(\sigma^\prime)\) freely. Thus
\[
\chi(\widetilde{O(\sigma^\prime)/\bsym T})=\frac{\chi(O(\sigma^\prime))}2=0.
\]

If \(\dim(\sigma^\prime)=1\), then \(O(\sigma^\prime)\) is defined by \(s\ne 0, t\ne 0, u=0\) on \(U_c\). Hence, \(F=\Sing(Y^\prime)\cap U_c\) is defined by \(u=0,s=-t,t\ne 0\), which is isomorphic to \(\AA^1-\{o\}\). As \(U_c/\bsym T\cong \AA^3_k/\diag(1,-1,-1)\), the exceptional set of \(\widetilde{O(\sigma^\prime)/\bsym T}\to O(\sigma^\prime)/\bsym T\) is \(F\times\PP_k^1\). Thus,
\[
\chi(\widetilde{O(\sigma^\prime)/\bsym T})=\frac{\chi(O(c)\backslash F)}2+\chi(F)\chi(\PP^1)=0.
\]

Therefore, we obtain
\[\sum_{\sigma^\prime\in(\Sigma^\prime-|\tau|)/\bsym S}\chi(\widetilde{O(\sigma^\prime)/\bsym T})=2\sharp((\Sigma^\prime_3-\{\tau\})/\bsym S)=2(r-1).
\]
From Theorem \ref{thm:subthm}, \(\chi(\widetilde{\AA^3/G^\prime})=6\).

From Lemma \ref{lem:conjnum2}, we obtain
\begin{align*}
\chi(\widetilde X)&=\frac{(r-1)(r-2)}6+2(r-1)+6\\
&=\frac{(r-1)(r-2)}6+2r+4\\
&=\mathrm{Conj}(G)+3.
\end{align*}
\end{proof}


\begin{thebibliography}{99}
\bibitem{Art} M. Artin, \textit{Covering of the rational double points in characteristic \(p\)}, Complex analysis and algebraic geometry, pp 11-22, 1977
\bibitem{Bat} V. V. Batyrev, \textit{Non-Archimedean integrals and stringy Euler numbers of log-terminal pairs}, J. Eur. Math. Soc. 1, pp 5-33, 1999
\bibitem{CW} H. E. A. E. Campbell, D. L. Wehlau, \textit{Modular Invariant Theory}, Encyclopaedia of Mathematical Sciences 139, Springer, 2011
\bibitem{CDG} Y. Chen, R. Du, Y. Gao, \textit{Modular quotient varieties and singularities by the cyclic group of order \(2p\)}, Comm. Algebra 48, no. 12, pp 5490-5500, 2020
\bibitem{W-Y} M. M. Wood, T. Yasuda, \textit{Mass formulas for local Galois representations and quotient singularities II: dualities and resolution of singularities}, Algebra \& Number Theory, 11, pp 817-840, 2017
\bibitem{Yama} T. Yamamoto, \textit{A counterexample to the McKay correspondence in positive characteristic}, Osaka Univ., Master thesis, 2018
\bibitem{Yama2} T. Yamamoto, \textit{Pathological phenomena in the wild McKay correspondence}, Osaka Univ., Doctor thesis, 2021
\bibitem{Yas} T. Yasuda, \textit{The \(p\)-cyclic McKay correspondence via motivic integration}, Comp. Math. 150, pp 1125-1168, 2014

\end{thebibliography}
\end{document}